\documentclass[pdflatex,sn-mathphys-num]{sn-jnl}

\usepackage{graphicx}%
\usepackage{multirow}%
\usepackage{amsmath,amssymb,amsfonts}%
\usepackage{amsthm}%
\usepackage{mathrsfs}%
\usepackage[title]{appendix}%
\usepackage{xcolor}%
\usepackage{textcomp}%
\usepackage{manyfoot}%
\usepackage{booktabs}%
\usepackage{algorithm}%
\usepackage{algorithmicx}%
\usepackage{algpseudocode}%
\usepackage{listings}%
\usepackage{amssymb,amscd,pb-diagram, amsthm}
\usepackage[all,arc]{xy}

\theoremstyle{thmstyleone}%
\newtheorem{theorem}{Theorem}
\newtheorem{proposition}[theorem]{Proposition}%
\newtheorem*{theorem1}{Main Theorem}
\newtheorem*{question1}{Question}
\newtheorem*{definition1}{Definition}
\newtheorem*{corollary1}{Corollary}
\newtheorem{corollary}{Corollary}
\newtheorem{lemma}{Lemma}

\theoremstyle{thmstyletwo}%
\newtheorem{remark}{Remark}%

\theoremstyle{thmstylethree}%

\begin{document}

\title[On the complement of nef divisors on projective manifolds]{On the complement of nef divisors on projective manifolds}


\author*{\fnm{Sergei} \sur{Feklistov} \email{sergef791@gmail.com}}


\abstract{
Let $X'$ be a complex projective manifold, $\dim X'>1$, $Z$ a connected analytic subset of codimension one which is the support of a nef effective Cartier divisor $D$ on $X'$, $X:=X'\setminus Z$. Let $\kappa(D)$ be the Iitaka dimension of $D$. We prove that $X$ is not Hartogs if and only if $D$ is abundant and $\kappa(D)=1$. In particular, $X$ is not Hartogs if and only if $X$ admits a proper fibration over an affine curve. }




\maketitle

\section{Introduction}

In this paper, we discuss complex analytic properties of complements of effective divisors on complex manifolds. Namely, we study the Hartogs extension property (see definition below) for $X'\setminus Z$, where $X'$ is a compact complex manifold and $Z\subset X'$ is a connected analytic subset of codimension one, usually the support of an effective divisor. This topic is naturally related to several others: the structure of neighborhoods of $Z$, the existence of holomorphic or regular functions near $Z$ or on $X'\setminus Z$, the existence of Hermitian metrics with certain positivity and smoothness properties, and foliation dynamics.

For relations between properties of Hermitian metrics on line bundles and the complex analytic structure of complements, see \cite{Brunella10, Horing, Koike,Koike25, Ohsawa}. For connections between Ueda theory, holonomy representation, and complex dynamics of foliations, see for example \cite{Claudon, Koike20}. On complements of nef divisors and their complex algebraic properties, see also \cite{Russo1, Russo2}.

Before proceeding, let us mention a classical question found in Hartshorne \cite[Chapter 6, Section 3, Problem 3.4]{Hartshorne}. Consider the complement $X'\setminus Z$, where $Z$ is an irreducible smooth curve on a complete smooth complex algebraic surface $X'$. Is the condition $Z^{2}\geq 0$ sufficient for $X'\setminus Z$ to be a Stein manifold?

If $Z^{2}<0$, then $Z$ can be contracted to a point. The complement $X'\setminus Z$ has only constant holomorphic functions, and $Z$ admits a fundamental system of strongly pseudoconvex neighborhoods; in particular, holomorphic functions near $Z$ are the same as on a 2-dimensional Stein variety \cite{Grauert}.

If $Z^{2}>0$, standard arguments show that $X'\setminus Z$ is a proper modification of a Stein variety; in particular, $X'\setminus Z$ has many regular and holomorphic functions, while $Z$ admits a fundamental system of strongly pseudoconcave neighborhoods, so holomorphic functions near $Z$ are only constants \cite{Suzuki}.

The case $Z^{2}=0$ is more interesting, as it relates to another phenomenon: the existence of Stein non-affine surfaces (Serre's example; see again \cite{Hartshorne}). Hartshorne also asked: let $Z\subset X'$ be an embedding of a smooth connected closed curve in a smooth projective surface, where $Z^{2}=0$ and $X'\setminus Z$ contains no complete curves. Does it follow that $X'\setminus Z$ is Stein non-affine? The general answer is negative; the first counterexample was due to Ogus \cite{Ogus}. Ueda investigated complex analytic properties of neighborhoods of curves $Z$ with $Z^{2}=0$ and classified them into four classes \cite{Ueda}. For example, if the Ueda type of $Z\hookrightarrow X'$ is finite (class $\alpha$ in Ueda's classification), then $X'\setminus Z$ is non-affine but a proper modification of a Stein variety; moreover, it has no non-constant regular functions but many holomorphic functions, and $Z$ admits a fundamental system of strongly pseudoconcave neighborhoods. If the Ueda type is infinite, then $Z$ may admit a fundamental system of pseudoflat neighborhoods (classes $\beta',\beta''$ in Ueda's classification). For more details on Ueda theory see also \cite{Neeman}.

Recall that a non-compact Stein manifold of dimension at least $2$ has the Hartogs extension property; thus the above questions about Steinness of $X'\setminus Z$ lead naturally to the weaker, purely function-theoretic property below. In this paper, we focus on the Hartogs extension property for complex manifolds.

\begin{definition1}
A non-compact connected complex manifold $X$ admits the \textbf{Hartogs phenomenon} (we often call such a manifold a \textbf{Hartogs} manifold) if for any compact set $K \subset X$ with $X\setminus K$ connected, the restriction homomorphism $\Gamma(X, \mathcal{O}_{X}) \to \Gamma(X\setminus K,\mathcal{O}_{X})$ is an isomorphism, where $\mathcal{O}_{X}$ is the sheaf of holomorphic functions.
\end{definition1}

For example, a proper modification of a Stein manifold of dimension at least 2 is Hartogs (on the Hartogs extension phenomenon for Stein spaces see \cite{Ov}).

More explicitly, we address the following general question:
\begin{question1}
Let $X'$ be a compact complex manifold, $D$ an effective divisor on $X'$ with connected support $Z$. What conditions on $D$ are necessary and sufficient for $X'\setminus Z$ to be Hartogs?
\end{question1}

For a smooth connected complex curve $Z$ on a compact (not necessarily projective) complex surface $X'$, we have the following statements by results of Grauert, Suzuki, Ueda and Corollary \ref{crucialcor}:

\begin{itemize}
\item If either $Z^{2}<0$, or $Z^{2}=0$ and of type $\beta'$, then $X'\setminus Z$ is not Hartogs;
\item If either $Z^{2}>0$, or $Z^{2}=0$ and of types $\alpha, \beta''$ or $\gamma$, then $X'\setminus Z$ is Hartogs.
\end{itemize}

For arbitrary effective divisors on nonsingular projective surfaces, see also Section \ref{conclusion}.

Now suppose $X'$ is a compact K\"ahler manifold, $\dim X'>2$, and $D$ is an effective divisor with connected support $Z$.

If the line bundle $\mathcal{O}(D)$ has a Hermitian metric whose curvature form is semipositive on Zariski tangent spaces of $Z$ everywhere and not identically zero at some point of $Z$, then $X'\setminus Z$ is Hartogs \cite[Theorem 0.2]{Ohsawa}. This also holds for nef effective divisors $D$ with numerical dimension $\nu(D)\geq 2$ \cite{Fek}; recall that the \textbf{numerical dimension} of a nef divisor $D$ is $\nu(D):=\max\{k\mid c_{1}(\mathcal{O}(D))^{k}\neq 0 \text{ in } H^{k,k}(X',\mathbb{R})\}$.

Assume $D$ is nef and $\nu(D)=1$. If $D=Z$ is an irreducible nonsingular hypersurface of a compact K\"ahler manifold $X'$ with topologically trivial normal bundle $N_{Z,X'}:=\mathcal{O}_{Z}(Z)$ and such that $\mathcal{O}(Z)$ admits a $C^{\infty}$-smooth Hermitian semipositive metric, there are results on the Hartogs property for $X'\setminus Z$ \cite[Theorem 1.1]{Koike}. Namely, if $N_{Z,X'}\in Pic^{0}(Z)$ is torsion, then $X'\setminus Z$ is not Hartogs because there exists a holomorphic fibration over a curve such that $Z$ is the support of a fibre. But if $N_{Z,X'}$ is non-torsion, then $X'\setminus Z$ is Hartogs. Note that $C^{\infty}$-semipositivity of $\mathcal{O}(Z)$ implies that the Ueda type of $Z\hookrightarrow X'$ is infinite. For an arbitrary effective divisor $ D$ with $C^{\infty}$-semipositive $\mathcal{O}(D)$, see also the recent work \cite{Koike25} (namely, Theorem 1.4). On the other hand, there exist nef line bundles that admit no smooth Hermitian metric with semipositive curvature \cite[Example 1.7]{PetDemSchn94} or \cite{Koike15}. 

Recall that the \textbf{Iitaka dimension} of an effective divisor $D$ is an integer $\kappa(D)\in\mathbb{Z}_{\geq 0}$ such that there exist $A,B\in\mathbb{R}_{>0}$ with $Am^{\kappa(D)}\leq h^{0}(X',\mathcal{O}(mD))\leq Bm^{\kappa(D)}$ for all $m\gg 0$ (for details see \cite[Chapter II, Section 3b]{Nak}). For any nef effective divisor $D$ on a projective manifold, $\kappa(D)\leq \nu(D)$ (\cite[Proposition 6.21]{Dem1}); $D$ is called \textbf{abundant} if $\kappa(D)=\nu(D)$. Obviously, for a nef effective divisor $D$ on a projective manifold, $\nu(D)>0$, and $\kappa(D)=0$ if and only if $h^{0}(X',\mathcal{O}(mD))=1$ for all $m\in\mathbb{Z}_{\geq 0}$.

The main result of this paper is the following necessary and sufficient condition for the Hartogs phenomenon for complements of nef effective divisors on projective manifolds.

\begin{theorem1}
Let $X'$ be a complex projective manifold, $\dim X'>1$, $D$ an effective nef divisor with connected support $Z$, $X:=X'\setminus Z$. Then $X$ is not Hartogs if and only if $D$ is abundant with Iitaka dimension one.
\end{theorem1}

Note that the Hartogs property depends only on $Z$, not on $D$. Hence, by the Main Theorem, if $D_1,D_2$ are two nef effective divisors with the same connected support $Z$, then $D_1$ is abundant with $\kappa(D_1)=1$ if and only if $D_2$ is; by the Corollary below, this holds exactly when $Z$ is the support of a fibre of a fibration over a projective curve.

The implication ``$D$ abundant with $\kappa(D)=1\Rightarrow X$ is not Hartogs'' is standard (Section \ref{case11}); it is based on \cite[Corollary 1]{Russo1} (see also Lemma \ref{Russo}) and on good properties of the Hartogs property under proper holomorphic maps (Section \ref{prophart}). The converse implication follows the same approach as \cite{Fek}, where we proved the Hartogs property for $\nu(D)>1$. Now we only need to consider the case $\kappa(D)=0$ and $\nu(D)=1$ (Section \ref{case01}). Specifically, we investigate the complete linear system $|-D|_{U}|$ associated with $(-D)|_{U}$ over a small neighborhood $U$ of $Z$ (Sections \ref{homologal} and \ref{linsys}).

Note that if $D=Z$ is a nonsingular hypersurface of $X'$ and $\mathcal{O}_{Z}(Z)$ is non-torsion with $\kappa(Z)=0,\nu(Z)=1$, then the Ueda type of $Z$ may be finite (see Remark \ref{uedarem}).

By the Main Theorem and Section \ref{case11} we obtain the following fibration existence criterion (where \textbf{fibration} means a surjective holomorphic map with connected fibers of dimension at least 1). For other fibration existence criteria see also \cite{Ding, Perera, Totaro}.

\begin{corollary1}
Let $X'$ be a complex projective manifold, $\dim X'>1$, $D$ an effective nef divisor with connected support $Z$, $X:=X'\setminus Z$. Then $X$ is not Hartogs if and only if there exist a fibration
$f\colon X'\to C'$ over a projective curve $C'$, a point $p\in C'$
and integers $m,k\in\mathbb{Z}_{>0}$ such that $mD=f^{*}(kp)$.
\end{corollary1}

In particular, we also obtain:

\begin{corollary1}\label{cor11}
Let $X'$ be a complex projective manifold, $\dim X'>1$, $Z$ a connected analytic subset of codimension one supporting an effective nef divisor, $X:=X'\setminus Z$. Then $X$ is not Hartogs if and only if there exists a proper fibration $X\to C$ over an affine curve.
\end{corollary1}

In Section \ref{examples} we consider an example where $D=Z$ is a nonsingular curve with $Z^{2}=0$ on a ruled surface.

In Section \ref{conclusion} we discuss the Question for effective but not necessarily nef divisors and the case of K\"ahler manifolds.

\section{Preliminaries}
In this section, we recall some classical results from Ueda theory and on the existence of fibrations. We also prove some auxiliary results needed for the proof of the main theorem and its corollaries.

\subsection{Ueda theory}\label{ueda}
In this section, we recall basic facts from Ueda theory (see \cite{Neeman, Claudon} for more details).

Let $X'$ be a K\"ahler manifold, $Z\subset X'$ a nonsingular compact complex hypersurface, and $U\subset X'$ a small neighborhood of $Z$ (so that $U$ and $Z$ have the same homotopy type). Assume that the normal bundle $N_{Z,X'}=\mathcal{O}_{Z}(Z)$ is topologically torsion; hence $N_{Z,X'}$ carries a flat unitary connection corresponding to a representation $\rho_{Z}\colon \pi_{1}(Z)\to S^{1}$. Since $U$ and $Z$ are homotopically equivalent, we obtain a line bundle $\widetilde{N_{Z,X'}}$ over $U$ with a flat unitary connection. The line bundle $\mathcal{O}_{U}(Z)$ is another extension of $N_{Z,X'}$. We call $\mathcal{U}=\mathcal{O}_{U}(Z)\otimes \widetilde{N_{Z,X'}}^{-1}$ the Ueda line bundle over $U$.

Let $\mathcal{I}\subset \mathcal{O}_{X'}$ be the ideal sheaf of $Z$, and let $Z(k)=\operatorname{Spec}(\mathcal{O}_{U}/\mathcal{I}^{k+1})$ be the $k$-th infinitesimal neighborhood of $Z$.

The Ueda type of $Z$, denoted by $\operatorname{utype}(Z)$, is defined as follows. If $\mathcal{U}|_{Z(l)}\cong \mathcal{O}_{Z(l)}$ for all $l<k$ but $\mathcal{U}|_{Z(k)}\not\cong \mathcal{O}_{Z(k)}$, then $\operatorname{utype}(Z)=k$. If $\mathcal{U}|_{Z(l)}\cong \mathcal{O}_{Z(l)}$ for all $l\in\mathbb{Z}_{\geq 0}$, then $\operatorname{utype}(Z)=\infty$.

\begin{proposition}\label{uedabase2}
Let $Z\subset U\subset X'$ and $Y\subset V\subset W'$ be as above. Consider the commutative square:
$$\xymatrix{
Z \ar@{->>}[d]^{f} \ar@{^{(}->}[r]^{i_{Z}} & U \ar[d]^{g}\\
Y \ar@{^{(}->}[r]^{i_{Y}} & V}$$
where $i_{Z}, i_{Y}$ are closed immersions, $g^{-1}(Y)=Z$, the normal bundle $N_{Y,V}$ is topologically torsion, and $f$ is surjective. Then
$$\operatorname{utype}(Z)=r\cdot \operatorname{utype}(Y),$$
where $r$ is the ramification index of $g$ along $Z$ (i.e., the ramification of $g$ along a generic transversal slice to $Y$).
\end{proposition}
\begin{proof}
\cite[Proposition 8.4, p. 61]{Neeman}.
\end{proof}

\begin{proposition}\label{uedabase3}
Suppose $X'$ and $Z$ are as above, with $X'$ compact, and $N_{Z,X'}$ is torsion of order $k$ with $\operatorname{utype}(Z)>k$. Then $\operatorname{utype}(Z)=\infty$ and $kZ$ is a fiber of a fibration on $X'$.
\end{proposition}
\begin{proof}
\cite[Theorem 5.1, p. 109]{Neeman} or \cite[Theorem 1.3]{Claudon}.
\end{proof}

In the projective case, this can be generalized as follows.

\begin{proposition}\label{uedabase4}
Let $X'$ be a projective complex manifold and $D$ an effective divisor with connected support. Suppose $\mathcal{O}_{D}(D)$ is torsion of order $k$. If $\mathcal{O}_{(k+1)D}(kD)\cong \mathcal{O}_{(k+1)D}$, then there exists a nonconstant morphism $f\colon X' \to C'$ to a smooth curve $C'$ with connected fibers which maps the divisor $D$ to a point.
\end{proposition}
\begin{proof}
\cite[Theorem 2.3]{Loray}.
\end{proof}

For projective surfaces, finiteness of the Ueda type of a smooth curve implies the following analytic properties of its complement.

\begin{proposition}\label{uedabase1}
Let $X'$ be a projective surface and $Z\subset X'$ a connected smooth projective curve with $\operatorname{utype}(Z)<\infty$. Then $X'\setminus Z$ is holomorphically convex and contains only finitely many compact curves $C$ (which can be contracted, in particular $C^{2}<0$). Moreover, $X'\setminus Z$ has no nonconstant regular functions, i.e., $\mathbb{C}[X'\setminus Z]=\mathbb{C}$.
\end{proposition}
\begin{proof}
\cite[Proposition 5.3, Corollary 5.4, Proposition 5.6]{Neeman}.
\end{proof}

\subsection{Lefschetz-type property}\label{homologal}

Let $X'$ be a compact complex manifold, $Z$ a connected analytic subset, $X:=X'\setminus Z$, and let $j\colon Z\hookrightarrow X'$, $i\colon X \hookrightarrow X'$ be the canonical closed and open embeddings, respectively. For any compact set $K\subset X$, let $i_{K}\colon X\setminus K \hookrightarrow X$ be the canonical open embedding.

Let $i^{-1}$, $j^{-1}$, $i_{K}^{-1}$ be the (topological) inverse image functors, $i_{!}$ the extension by zero functor, and $j_{*}$, $(i_{K})_{*}$ the direct image functors. Recall the following basic facts (see, for instance, \cite{Shapira} or \cite{Manin}):

\begin{enumerate}
\item For any sheaf $\mathcal{F}$ on $X'$, we have the exact sequence:
$$0\to i_{!}i^{-1}\mathcal{F}\to\mathcal{F}\to j_{*}j^{-1}\mathcal{F}\to 0.$$

\item For any injective sheaf $\mathcal{I}$ on $X$, we have the exact sequence:
$$0\to \Gamma_{K}(\mathcal{I})\to\mathcal{I}\to (i_{K})_{*}i_{K}^{-1}\mathcal{I}\to 0,$$
where $\Gamma_{K}(\mathcal{I})$ is the sheaf given by $U\mapsto\Gamma_{K}(U,\mathcal{I}):=\{f\in\Gamma(U,\mathcal{I})\mid \text{support of } f \text{ is contained in } K\cap U\}$.

\item For any sheaf $\mathcal{F}$ on $X$, we have the canonical isomorphism
$$H^{*}_{c}(X,\mathcal{F})\cong\varinjlim\limits_{K} H^{*}_{K}(X,\mathcal{F}),$$
where the colimit is taken over all compact sets $K\subset X$.
\end{enumerate}

This implies the following lemma.

\begin{lemma}\label{mainlemma}
Let $X', X, Z$ be as above, and let $\mathcal{F}$ be a sheaf of $\mathbb{C}$-vector spaces (not necessarily coherent) on $X'$ such that $i^{-1}\mathcal{F}\cong i^{-1}\mathcal{O}_{X'}=\mathcal{O}_{X}$ (as sheaves of $\mathbb{C}$-vector spaces) and $\dim H^{1}(X',\mathcal{F})=m <\infty$. The following assertions are equivalent:
\begin{enumerate}
\item $X$ is Hartogs;
\item The canonical homomorphism $\Gamma(X',\mathcal{F})\to \Gamma(Z,j^{-1}\mathcal{F})$ is surjective;
\item $\dim H^{1}_{c}(X,\mathcal{O}_{X})\leq m$.
\end{enumerate}
\end{lemma}

\begin{proof}
We have the commutative diagram:
\[
\begin{diagram}
\node{\Gamma(X',\mathcal{F})} \arrow{e,t}{} \arrow{s,r}{}
\node{\Gamma(Z,j^{-1}\mathcal{F})} \arrow{e,t}{} \arrow{s,r}{}
\node{H^{1}_{c}(X,\mathcal{O}_{X})} \arrow{s,r}{\cong} \arrow{e,t}{} 
\node{H^{1}(X',\mathcal{F})}\arrow{s,r}{}
\\
\node{\Gamma(X,\mathcal{O}_{X})} \arrow{e,t}{r} 
\node{\varinjlim\limits_{K\subset X}\Gamma(X\setminus K,\mathcal{O}_{X})} \arrow{e,t}{c}
\node{H^{1}_{c}(X,\mathcal{O}_{X})} \arrow{e,t}{}
\node{H^{1}(X,\mathcal{O}_{X})}
\end{diagram}
\]

Note that since $Z$ and $X$ are connected, $X$ is Hartogs if and only if the canonical homomorphism $r$ is an isomorphism. Also, the first square of the diagram is Cartesian because $i^{-1}\mathcal{F}\cong \mathcal{O}_{X}$.

The implications $1\Rightarrow 2\Rightarrow 3$ are clear via diagram chasing. The proof of $3\Rightarrow 1$ follows the argument in \cite[Corollary 4.3]{AndrHill}, which we recall below.

We may assume that $\varinjlim\limits_{K\subset X}\Gamma(X\setminus K,\mathcal{O}_{X})\neq \mathbb{C}$. Consider an equivalence class $f\in \varinjlim\limits_{K\subset X}\Gamma(X\setminus K,\mathcal{O}_{X})$ of a nonconstant holomorphic function on $X\setminus K$. Suppose $c(f)\neq 0$ (otherwise $f$ is in the image of $r$ by exactness). By assumption (3), $\dim H^{1}_{c}(X,\mathcal{O}_{X})\leq m$. Hence the $m+1$ elements $c(f), c(f^{2}),\dots, c(f^{m+1})$ are linearly dependent, and there exists a nonzero polynomial $P=\sum_{i=1}^{m+1}a_iT^{i}\in \mathbb{C}[T]$ without constant term such that $c(P(f))=0$. It follows that there is a holomorphic function $H\in\Gamma(X,\mathcal{O}_{X})$ such that $r(H)=P(f)$. Since $f$ is nonconstant, $P$ is nonconstant and $X\setminus K$ is connected, $H$ is nonconstant.

Now the $m+1$ elements $c(f), c(r(H)f), \dots, c(r(H)^{m}f)$ are linearly dependent. Hence, there exists a nonzero polynomial $P_{1}\in \mathbb{C}[T]$ such that $c(P_{1}(r(H))f)=0$. It follows that there exists a holomorphic function $F\in\Gamma(X,\mathcal{O}_{X})$ such that $r(F)=P_{1}(r(H))f$. Denoting $G=P_{1}(H)$, we obtain $r(F)=r(G)f$. Note that $G\not\equiv 0$: otherwise the nonconstant function $H$ would take only finitely many values on the connected manifold $X$, which is impossible.

Since $r(G^{m+1}H)=r(G^{m+1}P(F/G))$ and $G^{m+1}P(F/G)=\sum_i a_iF^{i}G^{m+1-i}$ is holomorphic on $X$, it follows that $G^{m+1}H=G^{m+1}P(F/G)$ on $X$. Hence, $H=P(F/G)$ on $X\setminus \{G=0\}$. Thus, $F/G\in \Gamma(X\setminus \{G=0\}, \mathcal{O}_{X})$ is a root of a polynomial of degree $\deg P\geq 1$ with constant nonzero leading coefficient and locally bounded other coefficients, so it is locally bounded near $\{G=0\}$. Since $G\not\equiv 0$, the Riemann extension theorem implies that $F/G\in \Gamma(X,\mathcal{O}_{X})$ and $r(F/G)=f$. Therefore, the canonical homomorphism $r$ is an isomorphism.
\end{proof}

In particular, we obtain the following Lefschetz-type property which can be viewed as a generalization of the Grauert-Grothendieck theorem \cite[Theorem 3.1]{Bost} or \cite[p. 83]{Gro} on the Lefschetz property for algebraic vector bundles over projective manifolds.

\begin{corollary}
Let $X', X, Z$ be as above. If $X$ is Hartogs, then for any sheaf of $\mathbb{C}$-vector spaces $\mathcal{F}$ on $X'$ such that $i^{-1}\mathcal{F}\cong \mathcal{O}_{X}$ and $\dim H^{1}(X',\mathcal{F})=m <\infty$, the canonical homomorphism $\Gamma(X',\mathcal{F})\to \Gamma(Z,j^{-1}\mathcal{F})$ is surjective.
\end{corollary}

The following corollary is our main tool.

\begin{corollary}\label{crucialcor}
Let $X'$ be a compact complex manifold, $D$ an effective divisor on $X'$ with connected support $Z$, and $X:=X'\setminus Z$. Then $X$ is Hartogs if and only if $\Gamma(Z,j^{-1}\mathcal{O}(-D))=0$.
\end{corollary}

\begin{proof}
Apply Lemma \ref{mainlemma} to $\mathcal{F}=\mathcal{O}_{X'}(-D)$. Then $i^{-1}\mathcal{F}\cong\mathcal{O}_{X}$ because $D$ is supported on $Z$, and $\dim H^{1}(X',\mathcal{F})<\infty$ because $\mathcal{F}$ is coherent and $X'$ is compact. Moreover, $\Gamma(X',\mathcal{O}_{X'}(-D))=0$: a holomorphic function on the compact connected manifold $X'$ vanishing along $D$ is identically zero. Hence the homomorphism in assertion (2) of Lemma \ref{mainlemma} is surjective if and only if $\Gamma(Z,j^{-1}\mathcal{O}(-D))=0$.
\end{proof}

\subsection{Moving divisors}\label{linsys}
Let $D$ be a divisor on a complex manifold. Consider the complete linear system associated with $D$:
$$|D|:=\{D'\in \operatorname{Div}(X)\mid D'\geq 0,\ D'\sim D\}.$$
We have the canonical bijection induced by $s\mapsto \operatorname{div}(s)$:
$$(\Gamma(X,\mathcal{O}(D))\setminus \{0\})/\Gamma(X,\mathcal{O}^{*})\cong |D|.$$

Geometrically, the condition $\Gamma(Z,j^{-1}\mathcal{O}(-D))=0$ in Corollary \ref{crucialcor} means the following. For an open neighborhood $U\supset Z$, consider the complete linear system
$$|-D|_{U}|:=\{D'\in \operatorname{Div}(U)\mid D'\geq 0,\ D'+D|_{U}\sim 0\}.$$

Taking the direct limit over all open neighborhoods $U\supset Z$, we obtain
$$\|-D\|:= \varinjlim\limits_{U\supset Z} |-D|_{U}|\cong (\Gamma(Z,j^{-1}\mathcal{O}(-D))\setminus \{0\})/\Gamma(Z,j^{-1}\mathcal{O}^{*}).$$

Hence the condition $\Gamma(Z,j^{-1}\mathcal{O}(-D))=0$ means that $\|-D\|=\emptyset$. Therefore,

\begin{lemma}\label{lemm22}
Let $X'$ be a compact complex manifold, $D$ an effective divisor on $X'$ with connected support $Z$, and $X:=X'\setminus Z$. Then $X$ is Hartogs if and only if $\|-D\|=\emptyset$.
\end{lemma}

\begin{remark}\label{remrigid}
The condition $\|-D\| = \emptyset$ means that $D$ is rigid in the following sense: there is no holomorphic map $f \colon U \to \Delta$ from a small neighbourhood $U$ of $Z$ to a small disc $\Delta\subset\mathbb{C}$ such that $D \hookrightarrow f^{-1}(0)$.
\end{remark}

Let $X'$ be a compact complex manifold, and let $Z\subset X'$ be a connected reduced analytic set of codimension one. Define the following poset of effective divisors whose support is exactly $Z$:
$$\operatorname{Div}_{Z}(X'):=\{D\in \operatorname{Div}(X')\mid D\geq 0,\ \operatorname{Supp}(D)=Z\}$$
with the partial order: $D_{2}\geq D_{1}$ if and only if $D_{2}-D_{1}$ is effective.

Lemma \ref{lemm22} implies the following corollary.

\begin{corollary}\label{lemm22'}
Let $X'$ be a compact complex manifold, $Z$ a connected reduced analytic set of codimension one, and $X:=X'\setminus Z$. Then the following are equivalent:
\begin{enumerate}
\item $X$ is Hartogs;
\item There exists $D\in \operatorname{Div}_{Z}(X')$ with $\|-D\|=\emptyset$;
\item For any $D\in \operatorname{Div}_{Z}(X')$, we have $\|-D\|=\emptyset$.
\end{enumerate}
\end{corollary}

If $D_{1},D_{2}\in \operatorname{Div}_{Z}(X')$ with $D_{2}> D_{1}$, then we have the canonical exact sequence
\begin{equation}\label{ex1}
0\to \mathcal{O}_{X'}(-D_{2})\to \mathcal{O}_{X'}(-D_{1})\to \mathcal{O}_{D_{2}-D_{1}}(-D_{1})\to 0,
\end{equation}
where $\mathcal{O}_{D_{2}-D_{1}}(-D_{1}):=(\mathcal{O}_{X'}/\mathcal{O}_{X'}(-(D_{2}-D_{1})))\otimes_{\mathcal{O}_{X'}} \mathcal{O}_{X'}(-D_{1})$.

Let $j\colon Z\hookrightarrow X'$ be the canonical embedding. Applying the topological inverse image functor $j^{-1}$ to (\ref{ex1}), we obtain the exact sequence of sheaves on $Z$:
$$0\to j^{-1}\mathcal{O}_{X'}(-D_{2})\to j^{-1}\mathcal{O}_{X'}(-D_{1})\to j^{-1}\mathcal{O}_{D_{2}-D_{1}}(-D_{1})\to 0$$
and the exact sequence:
\begin{equation}\label{ex2}
0\to \Gamma(Z, j^{-1}\mathcal{O}_{X'}(-D_{2}))\to \Gamma(Z, j^{-1}\mathcal{O}_{X'}(-D_{1}))\to \Gamma(Z, \mathcal{O}_{D_{2}-D_{1}}(-D_{1})).
\end{equation}

In particular, we have an inclusion $\|-D_{2}\|\hookrightarrow \|-D_{1}\|$ given by $G\mapsto G+(D_{2}-D_{1})$. Given a chain $\{D_{n}\}_{n>0}:=\{D_{n}\in \operatorname{Div}_{Z}(X')\mid D_{n}< D_{n+1},\ n\in\mathbb{Z}_{>0}\}$ in $\operatorname{Div}_{Z}(X')$, we obtain a descending filtration
\begin{equation}\label{ex3}
\cdots \hookrightarrow \|-D_{n+1}\|\hookrightarrow \|-D_{n}\|\hookrightarrow\cdots\hookrightarrow \|-D_{1}\|.
\end{equation}

\begin{lemma}\label{crucialcor2}
Let $X'$ be a compact complex manifold, $Z\subset X'$ a connected reduced analytic set of codimension one, $\{D_{n}\}_{n>0}$ a chain in $\operatorname{Div}_{Z}(X')$, and $X:=X'\setminus Z$. If in (\ref{ex3}) we have $\|-D_{n}\|=\|-D_{n-1}\|$ for all $n\geq 2$, then $\|-D_{1}\|=\emptyset$.
\end{lemma}

\begin{proof}
For any $n\in\mathbb{Z}_{>0}$, we have $\|-D_{n}\|=\|-D_{1}\|$. Assume there exists an effective divisor $G$ near $Z$ with $G\sim -D_{1}$ (i.e., there exists a nonzero holomorphic function $f$ near $Z$ with $G=\operatorname{div}(f)-D_{1}\geq 0$). Then for any $n\in\mathbb{Z}_{>0}$, there exists an effective divisor $G_{n}$ near $Z$ with $G_{n}\sim -D_{n}$ and $G_{n}+D_{n}-D_{1}=G$. Hence $G_{n}+D_{n}=G+D_{1}=\operatorname{div}(f)$. Thus $\operatorname{div}(f)-D_{n}\geq 0$ for all $n>0$, which is impossible.
\end{proof}

The following corollary follows from the exact sequence (\ref{ex2}) and Lemma \ref{crucialcor2}.

\begin{corollary}\label{crucialcor2'}
Let $X'$ be a compact complex manifold, $Z\subset X'$ a connected reduced analytic set of codimension one, $\{D_{n}\}_{n>0}$ a chain in $\operatorname{Div}_{Z}(X')$, and $X:=X'\setminus Z$. If for all $n\geq 2$ we have $\Gamma(Z,\mathcal{O}_{D_{n}-D_{n-1}}(-D_{n-1}))=0$, then $\|-D_1\|=\emptyset$ and $X$ is Hartogs.
\end{corollary}

From now until the end of this section, assume that $Z$ is a connected reduced analytic set of codimension one in a complex projective manifold $X'$ with all irreducible components nonsingular, and that $Div_Z(X')$ contains a nef divisor $D$.

\begin{lemma}\label{lemm24}
If $\mathcal{O}_{Z}(D)\cong \mathcal{O}_{Z}$ and $\|-D\|\neq\emptyset$, then any element of $\|-D\|$ is represented by an effective divisor $R$ in a small neighborhood $U$ of $Z$ such that the support of $R$ is contained in $Z$.
\end{lemma}

\begin{proof}
Consider an effective divisor $G=\operatorname{div}(f)-D$ near $Z$ for some nonzero $f\in \Gamma(Z, j^{-1}\mathcal{O}_{X'}(-D))$. Let $G=R+T$, where $R$ and $T$ are effective divisors such that $\operatorname{Supp}(R)\subset Z$ and no irreducible component $Z_{i}$ of $Z$ is a component of $T$. Then $T=\operatorname{div}(f)-D-R$. For any component $Z_{i}$ of $Z$, we have $\mathcal{O}_{Z_{i}}(-R)\otimes\mathcal{O}_{Z_{i}}(-D)\cong \mathcal{O}_{Z_{i}}(T)$ and $\mathcal{O}_{Z_{i}}(T)$ is effective because $Z_{i}$ is not a component of $T$.

Let $D=\sum\alpha_i Z_{i}$ with $\alpha_{i}>0$, and $R=\sum\beta_i Z_{i}$ with $\beta_{i}\geq 0$. Since $\mathcal{O}_{Z_{i}}(D)\cong\mathcal{O}_{Z_{i}}$, we have
\begin{equation}\label{eq1}
\mathcal{O}_{Z_{i}}(-\alpha_{i} Z_{i})\cong \mathcal{O}_{Z_{i}}\left(\sum\limits_{j:j\neq i}\alpha_{j}Z_{j}|_{Z_{i}}\right).
\end{equation}

Since $\mathcal{O}_{Z_{i}}(-R)\cong \mathcal{O}_{Z_{i}}(T)$, we have
\begin{equation}\label{eq2}
\mathcal{O}_{Z_{i}}(-\beta_{i} Z_{i})\otimes\mathcal{O}_{Z_{i}}\left(-\sum\limits_{j:j\neq i}\beta_{j}Z_{j}|_{Z_{i}}\right)\cong \mathcal{O}_{Z_{i}}(T|_{Z_{i}}).
\end{equation}

If $\beta_{i}=0$, then $(-T)|_{Z_{i}}$ is equivalent to an effective divisor. Since $T|_{Z_{i}}$ is effective, it follows that $T|_{Z_{i}}=0$. In particular, $\beta_{j}=0$ for all $j$ with $j\neq i$ and $Z_{j}\cap Z_{i}\neq \emptyset$. Since $Z$ is connected, we have $\beta_{j}=0$ for all $j$. In this case $R=0$ and $T=0$ in a small neighborhood of $Z$, so $D=\operatorname{div}(f)$ (i.e., $\|-D\|$ contains the zero divisor).

Assume $\beta_{i}\neq 0$ for all $i$. Then from (\ref{eq1}) and (\ref{eq2}), for any $i$ we obtain
$$\mathcal{O}_{Z_i}\left(\sum\limits_{j:j\neq i}\beta_{i}\alpha_{j}Z_{j}|_{Z_{i}}\right)\otimes \mathcal{O}_{Z_{i}}\left(-\sum\limits_{j:j\neq i}\alpha_{i}\beta_{j}Z_{j}|_{Z_{i}}\right)\cong \mathcal{O}_{Z_{i}}(\alpha_{i} T|_{Z_{i}}).$$

Thus,
$$\mathcal{O}_{Z_{i}}\left(\sum\limits_{j:j\neq i}(\beta_{i}\alpha_{j}-\alpha_{i}\beta_{j})Z_{j}|_{Z_{i}}\right)\cong \mathcal{O}_{Z_{i}}(\alpha_{i} T|_{Z_{i}}).$$

Let $A$ be a very ample divisor on $X'$ and $A_{i}=A|_{Z_{i}}$. Then 

$$\sum\limits_{j:j\neq i}(\beta_{i}\alpha_{j}-\alpha_{i}\beta_{j})Z_{j}A_{i}^{n-2}=\alpha_{i} T A_{i}^{n-2}$$

Summing in $i$ we obtain 

$$\sum\limits_{i,j:j\neq i}(\beta_{i}\alpha_{j}-\alpha_{i}\beta_{j})Z_{j}A_{i}^{n-2}=\sum\limits_{i}\alpha_{i} TA_{i}^{n-2}$$

Since $\beta_{i}\alpha_{j}-\alpha_{i}\beta_{j}=-(\beta_{j}\alpha_{i}-\alpha_{j}\beta_{i})$ and $Z_{j}A_{i}^{n-2}=Z_{i}A_{j}^{n-2}$ we obtain $\sum\limits_{i}\alpha_{i} TA_{i}^{n-2}=0$. 

It follows that $T|_{Z_{i}}=0$. Hence $T=0$ in a small neighborhood of $Z$, and $R=\operatorname{div}(f)-D$ is an effective divisor with support $Z$.
\end{proof}

\begin{lemma}\label{lemm24'}
If $\mathcal{O}_{Z}(D)$ is torsion of finite order $r$ and $\|-rD\|\neq\emptyset$, then any element of $\|-rD\|$ is represented by an effective divisor $R$ in a small neighborhood $U$ of $Z$ such that the support of $R$ is contained in $Z$.
\end{lemma}

\begin{proof}
Apply Lemma \ref{lemm24} to the divisor $rD$.
\end{proof}

\begin{lemma}\label{Russo}
A divisor $D$ on a complex projective manifold $X'$ is semiample if and only if it is nef, abundant, and its graded algebra is finitely generated over $\mathbb{C}$. Moreover, the graded algebra of a divisor $D$ with $\kappa(D)\leq 1$ is finitely generated over $\mathbb{C}$.
\end{lemma}
\begin{proof}
By the GAGA principle and \cite[Corollary 1]{Russo1}.
\end{proof}

\begin{lemma}\label{corol8}
If $\mathcal{O}_{Z}(D)$ is torsion of finite order $r$ and $\|-rD\|\neq\emptyset$, then $\kappa(D)=1$. Moreover, $D$ is semiample of Iitaka dimension $\kappa(D)=1$.
\end{lemma}

\begin{proof}
Note that $\nu(D)=1$. Indeed, $\mathcal{O}_{Z_{i}}(rD)\cong\mathcal{O}_{Z_{i}}$ for any irreducible component $Z_i$ of $Z$. By the projection formula, $c_{1}(\mathcal{O}_{X'}(D))\cdot c_{1}(\mathcal{O}_{X'}(Z_{i}))=0$ in $\mathbb{R}$-valued cohomology of $X'$, since $c_{1}(\mathcal{O}_{Z_{i}}(D))=0$ in $\mathbb{R}$-valued cohomology of $Z_i$. In particular, $(c_{1}(\mathcal{O}_{X'}(D)))^{2}=0$, and $\nu(D)=1$.

By Lemma \ref{lemm24'}, there exists an effective divisor $R$ supported on $Z$ such that $\mathcal{O}_{U}(D')\cong \mathcal{O}_{U}$, where $D':=R+rD$. In particular, $\mathcal{O}_{2D'}(D')\cong \mathcal{O}_{2D'}$. By Proposition \ref{uedabase4} (with $k=1$), there is a morphism $f\colon X'\to C'$ onto a smooth projective curve with connected fibres and $f(Z)=\{p\}$ a point. Write $f^{*}p=\sum_{j}m_jF_j$ with $m_j>0$, where $F_j$ are the irreducible components of $f^{-1}(p)$; every $Z_i$ is one of the $F_j$.

We show that $D'=c\,f^{*}p$ for some $c\in\mathbb{Q}_{>0}$; in particular $Z=f^{-1}(p)$. Let $A$ be a very ample divisor on $X'$ and $Q=q_{jk}:=-(F_j\cdot F_k\cdot A^{n-2})_{j,k}$. Then $Q$ is symmetric, $q_{jk}\leq 0$ for $j\neq k$, and $Qm=0$ since $\mathcal{O}(f^{*}p)|_{F_k}$ is trivial. As $f^{-1}(p)$ is connected, \cite[Lemma I.2.10]{BPV} gives $\ker Q=\mathbb{R}m$.
Write $D'=\sum_j d_jF_j$ with $d_j\ge 0$. Then $(-Qd)_j=D'\cdot F_j\cdot A^{n-2}$. This equals $0$ if $F_j\subset Z$, because $\mathcal{O}_U(D')$ is trivial, and is $\ge 0$ otherwise, because $D'$ is effective with support $Z$. On the other hand, $\sum_j m_j(-Qd)_j=-(Qm)^{T}d=0$ by symmetry of $Q$. Hence $Qd=0$, so $d\in\ker Q$, i.e.\ $d=c\,m$ with $c=a/b\in\mathbb{Q}_{>0}$. Then $bD'=f^{*}(ap)$. 

Since $f$ has connected fibers, we obtain $h^{0}(X',\mathcal{O}(lbD'))=h^{0}(C',\mathcal{O}(lap))$, which grows linearly in $l$. Hence $\kappa(D')=1$. Since $D\le D'\le tD$ for some $t\gg0$, we get $\kappa(D)=1$. In particular, $D$ is abundant with $\kappa(D)=\nu(D)=1$. By Lemma \ref{Russo}, $D$ is semiample of Iitaka dimension $1$.
\end{proof}

Lemma \ref{corol8} implies the following corollary on the Hartogs property.

\begin{corollary}\label{crucialcor5}
If $\mathcal{O}_{Z}(D)$ is torsion of finite order $r$, $\kappa(D)=0$, and $\nu(D)=1$, then $\|-rD\|=\emptyset$. In particular, $X$ is Hartogs.
\end{corollary}

For the case where $\mathcal{O}_{Z}(D)$ is non-torsion with $\nu(D)=1$, we have the following lemma.

\begin{lemma}\label{crucialcor5'}
If $\mathcal{O}_{Z}(D)$ is non-torsion with $\nu(D)=1$, then $\|-D\|=\emptyset$. In particular, $X$ is Hartogs.
\end{lemma}

\begin{proof}
Write $D=\sum\alpha_iZ_i$ with $\alpha_i>0$, and let $A$ be a very ample divisor on $X'$, $A_i:=A|_{Z_i}$. Since $D$ is nef and $\nu(D)=1$, we have $c_{1}(\mathcal{O}(D))^{2}=0$, hence $\sum_i\alpha_i\,D\cdot Z_i\cdot A^{n-2}=D^{2}\cdot A^{n-2}=0$; as every term is $\ge0$, we get $D\cdot Z_i\cdot A^{n-2}=0$ for all $i$.

Fix $m\geq1$ and let $0\neq s\in\Gamma(Z,\mathcal{O}_{D}(-mD))$. Let $D_s$ be the maximal effective divisor with $0\leq D_s\leq D$ such that $s$ is divisible by $D_s$ (i.e.\ $D_s$ is the vanishing subscheme of $s$), and let $s'$ be the resulting nonzero section of $\mathcal{O}_{D-D_s}(-(mD+D_s))$. Write $D_s=\sum\mu_iZ_i$ with $0\leq\mu_i\leq\alpha_i$, and put $I=\{i\mid\mu_i<\alpha_i\}$. Note that $I\neq\emptyset$, since otherwise $D_s=D$ and $s=0$.

For $i\in I$ the section $s'|_{Z_i}$ is not identically zero (otherwise $s$ would be divisible by $D_s+Z_i\leq D$, contradicting maximality), so
$\mathcal{O}_{Z_i}(-(mD+D_s))\cong\mathcal{O}_{Z_i}(E_i)$, where $E_i=\operatorname{div}(s'|_{Z_i})$ is an effective divisor on $Z_i$. Comparing degrees with respect to $A_i$ and using $D\cdot A_i^{n-2}=0$, we get
$$0\leq E_i\cdot A_i^{n-2}=-D_s\cdot A_i^{n-2},\qquad\text{i.e. } D_s\cdot A_i^{n-2}\leq0\quad (i\in I).$$
For $i\notin I$ we have $\mu_i=\alpha_i$, so
$$D_s\cdot A_i^{n-2}=D\cdot A_i^{n-2}+(D_s-D)\cdot A_i^{n-2}=\sum_{j:j\neq i}(\mu_j-\alpha_j)\,Z_j\cdot A_i^{n-2}\leq0 .$$
Hence $D_s\cdot A_i^{n-2}\leq0$ for all $i$, that is
\begin{equation}\label{eq3}
\mu_iZ_i\cdot A_i^{n-2}\leq-\sum_{j:j\neq i}\mu_jZ_j\cdot A_i^{n-2},
\end{equation}
while $D\cdot A_i^{n-2}=0$ gives
\begin{equation}\label{eq4}
\alpha_iZ_i\cdot A_i^{n-2}=-\sum_{j:j\neq i}\alpha_jZ_j\cdot A_i^{n-2}.
\end{equation}
Multiplying \eqref{eq3} by $\alpha_i$ and \eqref{eq4} by $\mu_i$ and subtracting, we obtain for every $i$
$$\sum_{j:j\neq i}\mu_i\alpha_jZ_j\cdot A_i^{n-2}\ \geq\ \sum_{j:j\neq i}\alpha_i\mu_jZ_j\cdot A_i^{n-2}.$$
Sum over $i$. Since $Z_j\cdot A_i^{n-2}=Z_j\cdot Z_i\cdot A^{n-2}=Z_i\cdot A_j^{n-2}$, interchanging the names of $i$ and $j$ shows that both sides of the summed inequality are equal. Hence equality holds for every $i$, so equality holds in \eqref{eq3} for every $i$:
$$D_s\cdot Z_i\cdot A^{n-2}=D_s\cdot A_i^{n-2}=0\quad\text{for all } i .$$
In particular $E_i\cdot A_i^{n-2}=0$ for $i\in I$; as $E_i$ is effective, $E_i=0$ for $i\in I$.

Let $i\notin I$. Then $0=D_s\cdot A_i^{n-2}=\sum_{j\in I}(\mu_j-\alpha_j)Z_j\cdot A_i^{n-2}$; since $\mu_j-\alpha_j<0$ and $Z_j\cdot A_i^{n-2}\geq 0$, so $Z_j\cap Z_i=\emptyset$ for all $j\in I$. Since $Z$ is connected and $I\neq\emptyset$, $I$ contains all components. Thus $E_i=0$ for all $i$, i.e.\ $s'|_{Z}$ vanishes nowhere and $\mathcal{O}_{Z}(-(mD+D_s))\cong\mathcal{O}_{Z}$.

Moreover $D_s\cdot Z_i\cdot A^{n-2}=0$ and $D\cdot Z_i\cdot A^{n-2}=0$ for all $i$. Apply \cite[Lemma I.2.10]{BPV} to $Q=-(Z_i\cdot Z_j\cdot A^{n-2})_{i,j}$: here $Q\alpha=0$ with $\alpha_i>0$ (condition (i)$'$), $q_{ij}\leq0$ for $i\neq j$, and condition (iii) holds because $Z$ is connected. Hence $\ker Q=\mathbb{R}\alpha$. Since $\mu\in\ker Q$, we get $\mu=c\alpha$ with $c\in\mathbb{Q}_{\geq0}$, say $c=p/q$, so $qD_s=pD$. Then
$$\mathcal{O}_Z\cong\mathcal{O}_Z(-q(mD+D_s))=\mathcal{O}_Z(-(qm+p)D),$$
so $\mathcal{O}_Z(D)$ is torsion, a contradiction. Therefore $\Gamma(Z,\mathcal{O}_{D}(-mD))=0$ for all $m\geq1$. Taking $D_n=nD$ (so $D_n-D_{n-1}=D$) in Corollary~\ref{crucialcor2'}, we obtain $\|-D\|=\emptyset$, and $X$ is Hartogs.
\end{proof}

\subsection{Proper maps with connected fibers}\label{prophart}

Let $f\colon X'\to Y'$ be a surjective proper holomorphic map between compact connected complex manifolds, $V\subset Y'$ a nonempty connected proper analytic subset, $Z=f^{-1}(V)$, $X:=X'\setminus Z$, and $Y:=Y'\setminus V$. The induced holomorphic map $f\colon X\to Y$ is surjective and proper. Assume that $f_{*}\mathcal{O}_{X}=\mathcal{O}_{Y}$.

\begin{lemma}\label{proplemma}
Let $X$ and $Y$ be as above. Then $X$ is Hartogs if and only if $Y$ is Hartogs.
\end{lemma}

\begin{proof}
($\Rightarrow$) Let $L\subset Y$ be compact with $Y\setminus L$ connected and let $g\in\Gamma(Y\setminus L,\mathcal{O}_{Y})$. Then $K:=f^{-1}(L)$ is compact, and $X\setminus K=f^{-1}(Y\setminus L)$ is connected. Since $X$ is Hartogs, $f^{*}g$ extends to $H\in\Gamma(X,\mathcal{O}_{X})=\Gamma(Y,\mathcal{O}_{Y})$, i.e. $H=f^{*}G$ with $G\in\Gamma(Y,\mathcal{O}_{Y})$. So, $G=g$ on $Y\setminus L$. Hence $Y$ is Hartogs.

($\Leftarrow$) Let $K\subset X$ be compact with $X\setminus K$ connected and let $h\in\Gamma(X\setminus K,\mathcal{O}_{X})$. The set $f(K)\subset Y=Y'\setminus V$ is compact. The connected set $V$ lies in a single connected component $\Omega$ of $Y'\setminus f(K)$. Since $V$ is a proper analytic subset of the connected complex manifold $\Omega$, the set $W:=\Omega\setminus V$ is connected and nonempty. Moreover $K_{Y}:=Y'\setminus\Omega$ is a compact subset of $Y$ with $Y\setminus K_{Y}=W$, and $W\subset Y\setminus f(K)$, so $f^{-1}(W)\subset X\setminus K$. Now
$$h|_{f^{-1}(W)}\in\Gamma(f^{-1}(W),\mathcal{O}_{X})=\Gamma(W,f_{*}\mathcal{O}_{X})=\Gamma(W,\mathcal{O}_{Y}),$$
so $h|_{f^{-1}(W)}=f^{*}g$ for some $g\in\Gamma(W,\mathcal{O}_{Y})$. Since $Y$ is Hartogs, $g$ extends to $G\in\Gamma(Y,\mathcal{O}_{Y})$. The functions $h$ and $f^{*}G$ are holomorphic on the connected set $X\setminus K$ and coincide on the nonempty open subset $f^{-1}(W)$, so $h=f^{*}G|_{X\setminus K}$. Hence $X$ is Hartogs.
\end{proof}

\section{Proof of Main Theorem}
Let $X'$ be a projective manifold, $D$ a nef effective divisor with connected support $Z$, and $X:=X'\setminus Z$. 

\subsection{Case $\kappa(D)=\nu(D)=1$}\label{case11}

Applying Lemma \ref{Russo}, $D$ is semiample, so $mD$ is base-point free for some $m\in\mathbb{Z}_{>0}$, and $\phi_{|mD|}$ has a curve as its image because $\kappa(D)=1$. Let $f\colon X'\to C'$ be the Stein factorization of $\phi_{|mD|}$: $C'$ is a smooth projective curve and $f$ has connected fibres. Then $mD=f^{*}H$ for an ample divisor $H$ on $C'$. Since $Z=f^{-1}(V)$ is connected, $V=\operatorname{Supp}H$ is a single point. Therefore $f$ is as in Section \ref{prophart} with $Y'=C'$, $V=\{p\}$ and $Y=C'\setminus\{p\}$. The curve $Y$ is not Hartogs. By Lemma \ref{proplemma}, $X$ is not Hartogs.

\begin{remark}
If $Z$ is a nonsingular irreducible divisor with $\kappa(Z)=\nu(Z)=1$, then the Ueda type $\operatorname{utype}(Z)=\infty$. Indeed, by Lemma \ref{Russo}, the linear system $|mZ|$ is base-point free. After Stein factorization, we may assume $\phi_{|mZ|}\colon X'\to Y$ has connected fibers. By Proposition \ref{uedabase2}, $\operatorname{utype}(Z)=r\cdot \operatorname{utype}(p)$, where $p\in Y$ is a point with $\phi_{|mZ|}^{*}(p)=mZ$. Since $\operatorname{utype}(p)=\infty$, it follows that $\operatorname{utype}(Z)=\infty$.
\end{remark}

\subsection{Case $\nu(D)\geq 2$}\label{case2}
The Kawamata--Viehweg vanishing theorem (see, e.g., \cite[Section 6.D]{Dem1} or \cite{Dem2}) implies that for any $m\in \mathbb{Z}_{>0}$, $\Gamma(Z,\mathcal{O}_{D}(-mD))=0$. Indeed, since $\Gamma(X',\mathcal{O}(-mD))=0$, the exact sequence $0\to\mathcal{O}(-(m+1)D)\to\mathcal{O}(-mD)\to\mathcal{O}_{D}(-mD)\to0$ gives an injection $\Gamma(Z,\mathcal{O}_{D}(-mD))\hookrightarrow H^{1}(X',\mathcal{O}(-(m+1)D))$. The space $H^{1}(X',\mathcal{O}(-(m+1)D))$   is dual to $H^{n-1}(X',K_{X'}+(m+1)D)$ by Serre duality, and this vanishes since $D$ is nef with $c_{1}(\mathcal{O}(D))^{2}\not=0$ \cite[Theorem 0.1]{Dem2}. Applying Corollary \ref{crucialcor2'} to the chain $D_{n}=nD$, it follows that $X$ is Hartogs.

\begin{remark}
In the projective case, one may also use a hyperplane section argument with induction on $\dim X'$ and the classical Kawamata-Viehweg vanishing theorem for nef and big divisors.
\end{remark}

\subsection{Case $\kappa(D)=0$, $\nu(D)=1$}\label{case01}
By a log resolution of $(X',Z)$ which is an isomorphism over $X'\setminus Z$ (so $X'\setminus Z$ is unchanged, and the support of the pull-back of $D$ is the preimage of $Z$, which is connected because the fibres of the resolution are connected), we may assume $D$ is an effective nef divisor and each irreducible component of $Z$ is nonsingular. The Iitaka and numerical dimensions remain unchanged under resolution (for the Iitaka dimension, see \cite[Chapter II, Section 3b, Remark 3.6]{Nak}; for the numerical dimension, note that $c_1(\pi^{*}D)^{k}=\pi^{*}\bigl(c_1(D)^{k}\bigr)$, and $\pi^{*}$ is injective on $H^{*}(X',\mathbb{R})$ because $\pi_{*}\pi^{*}=\mathrm{id}$ for a birational map $\pi$).

The sheaf $\mathcal{O}_{Z}(D)$ can be either non-torsion or torsion of finite order, and in both cases $X$ is Hartogs by Lemma \ref{crucialcor5'} and Corollary \ref{crucialcor5}.

\begin{remark}\label{uedarem}
If $Z$ is a nonsingular irreducible divisor with $\kappa(Z)=0$ and $\nu(Z)=1$, then $mZ$ cannot be a fiber of a fibration for any $m\in\mathbb{Z}_{>0}$. Since $\nu(Z)=1$, we have $(N_{Z,X'})^{\otimes m_0}\in \operatorname{Pic}^{0}(Z)$.
\begin{enumerate}
\item If $N_{Z,X'}$ is torsion, then $\operatorname{utype}(Z)<\infty$ by Proposition \ref{uedabase3};
\item If $N_{Z,X'}$ is non-torsion, then $\operatorname{utype}(Z)$ may be finite or infinite (see Proposition \ref{ruledprop}).
\end{enumerate}
\end{remark}

\subsection{Proof of the last corollary of the Introduction}\label{proofcor}
Suppose $X$ is not Hartogs. By the Main Theorem, $D$ is abundant with $\kappa(D)=1$, and Section \ref{case11} produces a proper fibration $f\colon X\to C'\setminus\{p\}$ over an affine curve.

Conversely, let $g\colon X\to C$ be a proper surjective holomorphic map with connected fibres onto an affine curve $C$ (a connected noncompact Riemann surface), and assume that $X$ is Hartogs. Since $C$ is smooth, $g_{*}\mathcal{O}_{X}=\mathcal{O}_{C}$, so $\Gamma(X,\mathcal{O}_{X})=\Gamma(C,\mathcal{O}_{C})$ via $g^{*}$. Let $R\subset C$ be a closed coordinate disc, centered at a point $a$. Then $C\setminus R$ is connected. Let $h$ be a meromorphic function on $C$ with a single simple pole at $a$ (it exists since $C$ is a noncompact Riemann surface). The set $K:=g^{-1}(R)$ is compact and $X\setminus K=g^{-1}(C\setminus R)$ is connected, because $g$ is proper with connected fibres. As $X$ is Hartogs, $g^{*}h$ extends to some $H\in\Gamma(X,\mathcal{O}_{X})=g^{*}\Gamma(C,\mathcal{O}_{C})$, i.e. $H=g^{*}G$ with $G\in\Gamma(C,\mathcal{O}_{C})$. By surjectivity of $g$, $G=h$ on $C\setminus R$, hence $G=h$ on $C\setminus\{a\}$ by the identity theorem, which contradicts the pole of $h$ at $a$. Thus $X$ is not Hartogs.

\section{Examples}\label{examples}

In this section, we consider ruled surfaces $X'$ over a complex curve $C$ and a section $Z$ with $Z^{2}=0$. For ruled surfaces, the reference \cite[Section 7, p. 51]{Neeman} is particularly useful. For the ruled surface $\pi\colon X'\to C$, there exists an extension:
\begin{equation}\label{ext}
0\to\mathcal{O}_{C}\to\mathcal{E}\to \mathcal{L}\to 0
\end{equation}
such that $\pi_{*}\mathcal{O}_{X'}(Z)\cong \mathcal{E}$ and $X'\cong \mathbb{P}_{C}(\mathcal{E})$ over $C$ (here $\mathbb{P}_{C}(-)$ denotes the projectivization of hyperplanes, not lines). Since $Z^{2}=0$, we have $\deg(\mathcal{L})=0$. Moreover, there exists an isomorphism $N_{Z,X'}=\mathcal{O}_{Z}(Z)\cong \mathcal{L}$.

\paragraph{Case of non-splitting of (\ref{ext})}
In this case, $\operatorname{utype}(Z)=1$ and $X'\setminus Z$ is a modification of a Stein variety (see \cite[Proposition 7.1, p. 51]{Neeman} and Proposition \ref{uedabase1}). In particular, $X'\setminus Z$ is Hartogs and $\kappa(Z)=0$.

\paragraph{Case of splitting of (\ref{ext})}
Then $X'\cong \mathbb{P}_{C}(\mathcal{O}_{C}\oplus\mathcal{L})$ over $C$ and $\operatorname{utype}(Z)=\infty$ (see \cite[Lemma 7.2, p. 51]{Neeman}, \cite[Hypothesis 5.2, p. 34]{Neeman}, and \cite[Definition 6.1, p. 9]{Neeman}). Moreover, $X'\setminus Z\cong \operatorname{Tot}(\mathcal{L}^{-1})$, and any small neighborhood $U$ of $Z$ is biholomorphic to a neighborhood of the zero-section of $\operatorname{Tot}(\mathcal{L})\to C$.

If $\mathcal{L}\in \operatorname{Pic}^{0}(C)$ is torsion of order $r$, then there exists an $r$-fold covering map $U\to C\times\Delta$, $(p,v)\mapsto (p,v^{\otimes r})$, where $\Delta=\{z\in\mathbb{C}\mid |z|<\varepsilon\}\subset \mathbb{C}$. Hence, $\Gamma(U,\mathcal{O}_{X'})\neq \mathbb{C}$: the function $v^{r}$ vanishes on $Z$, so it defines a nonconstant map $f\colon U\to\Delta$ with $Z\subset f^{-1}(0)$, i.e.\ $\|-Z\|\neq\emptyset$ (Remark~\ref{remrigid}). In particular, by Lemma \ref{lemm22}, $X'\setminus Z$ is not Hartogs and $\kappa(Z)=1$.

If $\mathcal{L}\in \operatorname{Pic}^{0}(C)$ is non-torsion, then $\Gamma(Z,\mathcal{O}_{Z}(-mZ))=0$ for all $m\in\mathbb{Z}_{>0}$. By Corollary \ref{crucialcor2'}, it follows that $X'\setminus Z$ is Hartogs and $\kappa(Z)=0$.

Summarizing, we obtain the following 
\begin{proposition}\label{ruledprop}
Let $\pi\colon X'=\mathbb{P}_{C}(\mathcal{E})\to C$ be a ruled surface corresponding to the exact sequence (\ref{ext}), and let $Z\subset X'$ be a section of $\pi$ with $\pi_{*}(\mathcal{O}_{X'}(Z))\cong\mathcal{E}$ and $Z^{2}=0$.
\begin{enumerate}
\item If (\ref{ext}) is non-split, then $\operatorname{utype}(Z)=1$, $\kappa(Z)=0$, and $X'\setminus Z$ is Hartogs;
\item If (\ref{ext}) is split, then $\operatorname{utype}(Z)=\infty$. Additionally:
\begin{enumerate}
\item If $\mathcal{L}$ is torsion, then $\kappa(Z)=1$ and $X'\setminus Z$ is not Hartogs;
\item If $\mathcal{L}$ is non-torsion, then $\kappa(Z)=0$ and $X'\setminus Z$ is Hartogs.
\end{enumerate}
\end{enumerate}
\end{proposition}

\section{Effective divisors that are not nef}\label{conclusion}

In this section, we discuss the Question stated in the Introduction for the case of effective but not necessarily nef divisors. This case is not fully understood; we only examine some specific instances here.

\paragraph{Nonsingular projective surfaces.}

Let $X'$ be a nonsingular projective surface, $D$ an effective divisor with connected support $Z$, and $X:=X'\setminus Z$. There exists a Zariski decomposition $mD = P + N$ for some $m \in \mathbb{Z}_{>0}$ (see, e.g., \cite{Nak}) such that:

\begin{itemize}
\item $P$ and $N$ are effective divisors (possibly zero);
\item $P$ is nef;
\item $N=\sum_i n_iN_i$, $n_i>0$, is exceptional, i.e.\ the intersection matrix $(N_i\cdot N_j)$ is negative definite;
\item $P \cdot N_i = 0$ for every irreducible component $N_i$ of $N$.
\end{itemize}
In particular, $Z=\operatorname{Supp}(P)\cup\operatorname{Supp}(N)$.

\begin{lemma}\label{Zardecom}
Let $X'$, $D$, $Z$, $P$, and $N$ be as above, and assume $P\neq0$. Then $\operatorname{Supp}(P)=Z$.
\end{lemma}

\begin{proof}
Write $P=\sum_j p_jP_j$ with $p_j>0$. Let $N_i$ be a component of $N$ with $N_i\not\subset\operatorname{Supp}(P)$. Since $P$ is nef and $P\cdot N_i=0$, we get
\[
0=P\cdot N_i=\sum_j p_j\,(P_j\cdot N_i),
\]
where $P_j\neq N_i$ for all $j$, so every term is $\geq0$. Hence $P_j\cdot N_i=0$, i.e.\ $N_i\cap P_j=\emptyset$ for all $j$. Thus $N_i\cap\operatorname{Supp}(P)=\emptyset$.

Let $S$ be the union of all components $N_i\not\subset\operatorname{Supp}(P)$. Then $S$ is closed, $S\cap\operatorname{Supp}(P)=\emptyset$, and $Z=\operatorname{Supp}(P)\sqcup S$ (the remaining components of $N$ lie in $\operatorname{Supp}(P)$). Since $Z$ is connected and $\operatorname{Supp}(P)\neq\emptyset$, we obtain $S=\emptyset$. Therefore $\operatorname{Supp}(N)\subset\operatorname{Supp}(P)$ and $\operatorname{Supp}(P)=Z$.
\end{proof}

\begin{proposition}\label{surfprop}
Let $X'$ be a nonsingular projective surface, $D$ an effective divisor with connected support $Z$, and $X := X'\setminus Z$. Let $mD = P + N$ be the Zariski decomposition of $mD$ for some $m$. Then $X$ is not Hartogs if and only if either $P=0$, or $P$ is abundant with Iitaka dimension $\kappa(P) = 1$.
\end{proposition}

\begin{proof}
\emph{Case $P=0$.} Then $Z=\operatorname{Supp}(N)$ is connected and exceptional, so by Grauert's contraction theorem there is a contraction $\pi\colon X'\to Y'$ onto a normal compact surface, and $\pi(Z)=\{q\}$ is a single point. Hence $X\cong Y'\setminus\{q\}$. By the Hartogs extension theorem on normal spaces, every holomorphic function on $X$ extends to $Y'$, so $\Gamma(X,\mathcal{O}_X)=\Gamma(Y',\mathcal{O}_{Y'})=\mathbb{C}$. Let $B\ni q$ be a small open neighbourhood of $q$ such that $B\setminus\{q\}$ is connected, and put $K:=Y'\setminus B$. Then $K\subset X$ is compact and $X\setminus K=B\setminus\{q\}$ is connected. Since $\dim B\geq2$, there exist nonconstant holomorphic functions on $B\setminus\{q\}$ (restrictions of local coordinates on $B$), and they do not belong to the image of $\Gamma(X,\mathcal{O}_X)=\mathbb{C}$. Hence $X$ is not Hartogs, as claimed.

\emph{Case $P\neq0$.} By Lemma~\ref{Zardecom}, $\operatorname{Supp}(P)=Z$. Thus $P$ is a nef effective divisor with connected support $Z$, and $X=X'\setminus\operatorname{Supp}(P)$. By the Main Theorem, $X$ is not Hartogs if and only if $P$ is abundant with $\kappa(P)=1$.
\end{proof}

We expect that the same strategy works for effective divisors on projective manifolds that admit a Zariski decomposition in the above sense; this would require an analogue of Lemma~\ref{Zardecom}, which we do not prove here.

\paragraph{K\"ahler manifolds.}
Note that the case $\nu(D) \geq 2$ (Subsection \ref{case2}) also holds for arbitrary compact K\"ahler manifolds. We conjecture that a purely analytic proof of the Main Theorem exists, even for arbitrary compact K\"ahler manifolds.

Since effective divisors admit a divisorial Zariski decomposition in the sense of Boucksom \cite[Section 19.B, Theorem 19.12]{Dem1}, we believe this may help extend the results to the general case using purely analytic methods.

\backmatter

\bmhead{Acknowledgements} The study has been funded within the framework of the HSE University Basic Research Program (September 2025 -- August 2026).


\bibliography{sn-bibliography}

\end{document}